\documentclass[a4paper,twoside]{article}      
\usepackage{amsmath,amssymb,amsfonts,amsthm,amscd,mathtools}
\usepackage{mdframed}
\usepackage{graphics}                 
\usepackage{color}                    
\usepackage{hyperref}                
\usepackage{ mathrsfs }
\usepackage{indentfirst}
\usepackage{bbold}
\usepackage{bm}
\usepackage{blkarray}
\usepackage{diagbox}
\usepackage{enumerate}
\usepackage{url}         
\usepackage{colonequals} 
\usepackage{authblk} 
\usepackage{a4wide}
\usepackage{graphicx}

\hypersetup{
    bookmarks=true,         
    unicode=false,          
    pdftoolbar=true,        
    pdfmenubar=true,        
    pdffitwindow=false,     
    pdfstartview={FitH},    
    pdftitle={My title},    
    pdfauthor={Author},     
    pdfsubject={Subject},   
    pdfcreator={Creator},   
    pdfproducer={Producer}, 
    pdfkeywords={keywords}, 
    pdfnewwindow=true,      
    colorlinks=true,       
    linkcolor=blue,          
    citecolor=blue,        
    filecolor=magenta,      
    urlcolor=cyan           
}

\newtheorem{theorem}{Theorem}[section]
\newtheorem{proposition}[theorem]{Proposition}

\newtheorem{lemma}[theorem]{Lemma}
\theoremstyle{definition}
\newtheorem{remark}[theorem]{Remark}

\newtheorem{example}{Example}[section]

\def\!{\mathop{\mathrm{!}}}

\def\N{\mathsf{N}}
\def\Nb{\mathbb{N}}

\def\Pp{\mathbf{P}}

\def\R{\mathbf{ R}}
\def\C{\mathcal{ C}}
\def\A{\mathcal{ A}}
\def\Var{\mathrm {Var}}
\def\P{\mathcal{P}}
\def\E{\mathbf{E}}
\def\x{\textbf{x}}
\def\d{\textbf{d}}
\def\j{\textbf{j}}

\newlength{\boxwidth}
\author[1]{Manh Hong Duong}
\author[2]{The Anh Han}
\affil[1]{School of Mathematics, University of Birmingham, UK}
\affil[2]{School of Computing, Engineering and Digital Technologies, Teesside University, UK}
\title{Random evolutionary games and random polynomials}
\begin{document}
\maketitle
\begin{abstract}
In this paper, we discover  that the class of random polynomials arising from the  equilibrium analysis of random asymmetric  evolutionary games is \textit{exactly} the Kostlan-Shub-Smale system of random polynomials, revealing an intriguing connection between evolutionary game theory and the theory of random polynomials. Through this connection, we analytically characterize the statistics of the number of internal equilibria of random asymmetric  evolutionary games, namely its mean value, probability distribution, central limit theorem and universality phenomena. 
Biologically, these quantities enable prediction of  the levels of social and biological diversity as well as the overall complexity in a dynamical system. 
By comparing symmetric and asymmetric random games, we  establish that symmetry in  group interactions  increases the expected number of internal equilibria. 
Our research establishes new theoretical understanding of  asymmetric evolutionary games and highlights the significance of symmetry and asymmetry in group interactions.   
\end{abstract}
\section{Introduction}
Statistics of roots of (systems of) random polynomials has become an active topic of research over the past century, dating back  to several  seminal papers~\cite{BP32,LittlewoodOfford1939,Kac1943,LittlewoodOfford1945,LittlewoodOfford1948}. The topic provides an everlasting source of challenging mathematical problems, driving  the developments of powerful methods and techniques in analysis, combinatorics and probability theory; see recent papers~\cite{TaoVu14,DoNguyenVu2015,DoNguyenVu2018,NNV2016, NguyenVu2021,NguyenVu2022} and references therein for the latest results of the field. It has also found applications in the study of complex phenomena/systems in many other disciplines, such as quantum chaotic dynamics \cite{Bogomolny1992} and quantized vortices in the ideal Bose gas \cite{castin2006quantized}) in physics, the theory
of computational complexity \cite{Shub1993}, feasibility and stability of ecological systems~\cite{may:1973ab,aladwani2020}, persistence and first-passage properties in non-equilibrium systems~\cite{SM07,SM08,bray2013persistence}, steady states of chemical reaction networks \cite{feliu2022kac} and the gradients of deep linear networks \cite{mehta2021loss}.  An important class of random polynomials intensively studied  in the literature is the Kostlan-Shub-Smale (also known as elliptic or binomial) random polynomials in which the variance of the random coefficients are binomials, cf. Section \ref{sec: KSS}. According to \cite{EK95}, ``this particular random polynomial is probably the more natural definition of a random polynomial''. In the present work, we show that it arises naturally yet from \textit{evolutionary game theory}.

Evolutionary game theory (EGT), which incorporates game theory into Darwin’s evolution theory, constitutes a powerful mathematical framework for the study of dynamics of frequencies of competing strategies in large populations. Introduced in 1973 by Maynard Smith and Price \cite{smith1973logic}, over the last 50 years, the theory has found its applications in diverse disciplines including biology, physics, economics, computer sciences and mathematics, see e.g. \cite{maynardsmith:1982to,nowak1992evolutionary, hofbauer1998evolutionary,nowak1992evolutionary, szabo2007evolutionary,santos:2011pn,HanBook2013} and the recent survey \cite{traulsen2023future} for more information. Incorporating stochasticity or randomness into evolutionary games is crucial for capturing the inherent uncertainty characteristic of complex systems. This uncertainty arises due to environmental and demographic noise and may also result from factors like insufficient data for measuring payoffs or unavoidable human estimation errors. \cite{may:1973ab, allesina2015stability,constable,bianconi2023complex}. The classical approach to evolutionary games is replicator dynamics  \cite{taylor:1978wv,zeeman:1980ze,hofbauer:1998mm,schuster:1983le,nowak:2006bo}, describing that whenever a strategy has a fitness larger than the  average fitness of the population, it is expected to  spread. The number of equilibrium points of the replicator dynamics and their stability provide valuable insights into  the evolutionary processes,  including predicting the levels of social, cultural or biological diversity and understanding the co-existence of different types in a population and the maintenance of polymorphism~\cite{gokhale:2010pn,galla2013complex, su2019evolutionary, huang2015stochastic}. 
In multi-player multi-strategy random evolutionary games, where the payoff entries are random variables, finding an equilibrium point consists in solving a system of multivariate random polynomials, and the number of equilibrium points is a (discrete) random variable.

Herein we show that the class of random polynomials arising from the study of equilibria of random \textit{asymmetric}  evolutionary games---where a player's payoff within a group  depends on the ordering of its members---corresponds \textit{exactly} to the celebrated Kostlan-Shub-Smale system of random polynomials. This is intriguing since in previous works, the analysis  of internal equilibria of \textit{symmetric} random evolutionary games, where a player's payoff in group interactions is independent of its members' ordering, results in a different class of random polynomials~\cite{gokhale:2010pn,HTG12,DuongHanJMB2016,DuongTranHanDGA,DuongTranHanJMB,can2022expected,duong2019persistence}. Using this connection, we characterize the fundamental statistical properties of the number of internal equilibria of random  asymmetric  evolutionary games, including the mean, the variance, the probability distribution, as well as a central limit theorem. While the mean number provides valuable information about the average macroscopic behavior concerning the number of internal equilibria a dynamical system might possess, the probability distribution offers further details into the likelihood of various states of biodiversity occurring within the system. The central limit theorem, which is a key concept in probability theory, establishes that under an appropriate re-scaling, the variance of the number of internal equilibria converges to a normal distribution. It is noteworthy that one of the most significant advances in equilibrium analyses in EGT and population genetics has been the study of the maximal number of equilibrium points of a system and the attainability of the patterns of evolutionarily stable strategies in an evolutionary system \cite{may:1973ab,karlin:1980aa,cannings1988patterns,karlin:1970tp,broom:1993pa,altenberg:2010tp,gokhale:2010pn,DuongTranHanJMB}. As a consequence of our analysis, we provide an explicit formula for the probability of obtaining the number of maximal number of internal equilibria in evolutionary games (see Theorem 3.3).  Moreover, we prove  that, on  average, symmetry enhances the  number of internal equilibria.
This has an important biological interpretation: symmetry increases the expected number of internal equilibria, and hence, the biological or behavioural diversity, of the evolutionary process. Furthermore, we show a universality phenomenon for asymmetric games, that is, asymptotically, the expectation of the number of equilibria does not depend on the specific distribution of the payoff entries. Inspired by this interesting result, we also numerically investigate and make conjectures on the universality properties for the number of internal equilibria of symmetric random evolutionary games. The main results of the paper are summarized in the (yellow) box below.

\textbf{Organization}. The rest of the paper is organized as follows. In Section \ref{sec: replicator} we recall the replicator dynamics for multi-player multi-strategy evolutionary games and the Kostlan-Shub-Smale system of random polynomials deriving the aforementioned connection. In Section \ref{sec: statistics} we characterize the statistics of the number of internal equilibria for random asymmetric evolutionary  games. In Section \ref{sec: sym vs asym} we compare symmetric and asymmetric games. We provide further discussions for future work in Section \ref{sec: summary}. Finally, technical results and detailed calculations are given in the Supporting Information (SI).
\newpage
\begin{mdframed}[backgroundcolor=yellow, linecolor=black,font={\sffamily},frametitle={\large{\underline{Statistics of the number of equilibria in asymmetric  games}}}]
Let $\mathcal{N}_{d,n}$ be the number of internal equilibria of $d$-player $n$-strategy asymmetric evolutionary games. Below for the expectation, variance, and probability distributions, we assume the payoff entries are i.i.d Gaussian random variables, while the universality phenomena does not require this Gaussian assumption.
\begin{enumerate}[(1)]
    \item \textbf{(Expectation)} The expected number of internal equilibria is 
\begin{equation*}
\mathbb{E}(\mathcal{N}_{d,n})=\frac{1}{2^{n-1}}(d-1)^{\frac{n-1}{2}}.
\end{equation*}
\item \textbf{(Variance)} The variance of the number of the internal equilibria satisfies the following asymptotic behaviour:
\begin{equation*}
\lim_{d\to\infty}\frac{4^{n-1}\rm{Var}(\mathcal{N}_{d,n})}{(d-1)^\frac{n-1}{2}}=V_\infty^2,
\end{equation*}
where $0<V_\infty<\infty$ is an explicit constant. Furthermore, $\mathcal{N}_{d,n}$ satisfies a central limit theorem, that is
\begin{equation*}
\frac{4^{n-1}\mathcal{N}_{d,n}-(d-1)^{\frac{n-1}{2}}}{(d-1)^\frac{n-1}{4}}
\end{equation*}
converges in distribution, as $d\rightarrow \infty$, to a normal random variable with positive variance.  
\item \textbf{(Probability distribution of $\mathcal{N}_{d,2}$)} The probability that a $d$-player two-strategy asymmetric random evolutionary game has $m$ ($0\leq m\leq d-1$) internal equilibria is
\begin{equation*}
p_{m}=\sum_{k=0}^{\lfloor \frac{d-1-m}{2}\rfloor}p_{m,2k,d-1-m-2k},
\end{equation*}
where $p_{m,2k,d-1-m-2k}$ are explicitly given in Section \ref{sec: distribution}.
\item \textbf{(Universality phenomena)} Suppose that the payoff entries are independent with mean $0$, variance $1$ and finite $(2+\varepsilon)$-moment for some $\varepsilon>0$. Then
\begin{equation*}
\mathbb{E}(\mathcal{N}_{d,2})=\frac{\sqrt{d-1}}{2}+O((d-1)^{1/2-c}),
\end{equation*}
for some $c>0$ depending only on $\varepsilon$.
\end{enumerate}
\end{mdframed}

\section{Multi-player multi-strategy games and random polynomials}
\label{sec: replicator}
\subsection{The replicator dynamics}
The classical approach to evolutionary games is replicator dynamics \cite{taylor:1978wv,zeeman:1980ze,hofbauer:1998mm,schuster:1983le,nowak:2006bo}, capturing Darwin's principle of natural selection that  whenever a strategy has a fitness larger than the average fitness of the population, it is expected to spread. In the present work, we consider \textit{asymmetric games} where the order of the participants is  relevant. As discussed in \cite{McAvoy2015}, ``Biological interactions, even between members of the same species, are almost always asymmetric due to differences in size, access to resources, or past interactions." Asymmetry also plays a crucial role in social, economic and multi-agent interactions due to the difference in roles and locations of the parties involved, see e.g. \cite{samuelson1992evolutionary, friedman1998economic,McAvoy2015,tuyls2018symmetric,Hauert2019,Sue2022,mcavoy2022evaluating,ogbo2022evolution}. Models using asymmetric games, instead of symmetric ones, are thus more realistic and  representative of  real-world interactions.

To describe the mathematical model, we consider an infinitely large population with $n$ strategies whose frequencies are denoted by $x_i$, $1 \leq i \leq n$. The frequencies are non-negative real numbers summing up to $1$, i.e. $\sum^n_{i=1} x_i = 1$. The interaction of the individuals in the population takes place in randomly selected groups of $d$ participants, that is, they play and obtain their fitness from $d$-player games. The fitness of a player is  calculated as average of the payoffs that they achieve from the interactions using a theoretic game approach. 
Let $i_0$, $1 \leq i_0 \leq n$, be the strategy of the focal player. Let $\alpha^{i_0}_{i_1,\ldots,i_{d-1}}$ be the payoff that the focal player obtains when it interacts with the group $(i_1,\ldots,i_{d-1})$ of $d-1$ other players where $i_k$ (with $1 \leq i_k \leq n$ and $1 \leq k \leq d-1$) be the strategy of the player in position $k$. Then the average payoff or fitness of the focal player is given by  
\begin{equation}
\label{eq: fitness 1}
\pi_{i_0}= \sum\limits_{1\leq i_1,\ldots,i_{d-1}\leq n}\alpha^{i_0}_{i_1,\ldots,i_{d-1}}x_{i_1}\ldots x_{i_{d-1}}.
\end{equation}
Given a set of non-negative integer numbers $\{k_i\}_{i=1}^{n}$ satisfying $\sum_{i=1}^n k_i=d-1$, let us define
\begin{multline*}
\A_{k_1,\ldots, k_n}:=\Big\{\{i_1,\ldots, i_{d-1}\}:~~1\leq i_1,\ldots, i_{d-1}\leq n~~\\ \text{and there are $k_i$ players using strategy $i$ in $\{i_1,\ldots, i_{d-1}\}$}\Big\}.
\end{multline*}
By the multinomial theorem, it follows that 
$$
|\A_{k_1,\ldots, k_n}|=\begin{pmatrix}
d-1\\k_1,\ldots,k_n
\end{pmatrix} =\frac{(d-1)!}{k_1!\ldots k_n!}.
$$
By re-arranging appropriate terms, Equation \eqref{eq: fitness 1} can be re-written as
\begin{equation}
\label{eq: fitness 2}
\pi_{i_0}=\sum_{\substack{0\leq k_1,\ldots, k_n\leq d-1\\ \sum\limits_{i=1}^{n} k_i=d-1}} a^{i_0}_{k_1,\ldots, k_n}\prod_{k=1}^n x_i^{k_i} \qquad\text{for}~~ i_0=1,\ldots, n,
\end{equation}
where
\begin{equation}
a^{i_0}_{k_1,\ldots,k_n}:=\sum_{\{i_1,\ldots,i_{d-1}\}\in \A_{k_1,\ldots, k_n}}\alpha^{i_0}_{i_1,\ldots,i_{d-1}}.
\end{equation}
Now  the replicator equations for $d$-player $n$-strategy games can be written as a system of $n-1$ differential equations \cite{hofbauer:1998mm,sigmund:2010bo} 
\begin{equation} 
\label{eq:replicator_eqs_2playersGame}
\dot{x_i} = x_i \left(\pi_i - \langle \pi \rangle \right) \quad\quad \text{ for } i = 1,\dots,n-1,
\end{equation} 
where $\langle \pi \rangle = \sum^n_{k = 1} x_k \, \pi_k$ is the average payoff of the population. Note that, in addition to the $n-1$ equations above,  $\sum_{i=1}^n x_i=1$ must also be satisfied.
\subsection{Equilibria of the replicator dynamics}
It follows from \eqref{eq:replicator_eqs_2playersGame} that the vertices of the unit cube in $\R^n$ are  equilibria of the replicator dynamics. In the following analysis, we focus on \textit{internal equilibria}, which are given by the points $(x_1, \dots, x_n)$ where $0 < x_i < 1$ for all $1 \leq i \leq n-1$ that satisfy
\[
\pi_i=\langle \pi\rangle \qquad\text{for all}~~i=1,\ldots, n.
\]
The  system above is equivalent to $\pi_i-\pi_n=0$ for all $i=1,\ldots, n-1$. Using \eqref{eq: fitness 2} we obtain  a system of $n-1$ equations of multivariate polynomials of degree $d-1$
\begin{equation}
\label{eq: eqn for fitness2}
\sum\limits_{\substack{0\leq k_{1},\ldots, k_{n}\leq d-1,\\ \sum^{n}\limits_{i = 1}k_i = d-1  }}b^i_{k_{1}, \ldots,k_{n-1} }\prod\limits_{i=1}^{n}x_{i}^{k_i} = 0\quad \text{ for } i = 1, \dots, n-1,
\end{equation}
where 
\begin{align}
b^{i}_{k_{1},\ldots, k_{n} }& := a^{i}_{k_{1},\ldots, k_{n} } -a^{n}_{k_{1},\ldots, k_{n} }\notag
\\&=\sum_{\{i_1,\ldots,i_{d-1}\}\in \A_{k_1,\ldots, k_n}}\Big(\alpha^{i}_{i_1,\ldots,i_{d-1}}-\alpha^{n}_{i_1,\ldots,i_{d-1}}\Big)\notag
\\&=\sum_{\{i_1,\ldots,i_{d-1}\}\in \A_{k_1,\ldots, k_n}}\beta^{i}_{i_1,\ldots,i_{d-1}},
\label{eq: b's}
\end{align}
where $\beta^{i}_{i_1,\ldots,i_{d-1}}$ denotes the difference of the payoff entries
\begin{equation}
 \label{eq: beta}   
\beta^{i}_{i_1,\ldots,i_{d-1}}:=\alpha^{i}_{i_1,\ldots,i_{d-1}}-\alpha^{n}_{i_1,\ldots,i_{d-1}}.
\end{equation}
Using the transformation $y_i = \frac{x_i}{x_n}$ (recalling that $0<x_n<1$), with $0< y_i < +\infty$ and $1 \leq i \leq n-1$ and dividing the left hand side of ~\eqref{eq: eqn for fitness2} by $x_n^{d-1}$ we obtain the following system of polynomial equations in terms of $(y_1,\ldots,y_{n-1})$
\begin{equation}
\label{eq: eqn for fitnessy}
 \sum\limits_{\substack{0\leq k_{1},\ldots, k_{n-1}\leq d-1,\\  \sum\limits^{n-1}_{i = 1}k_i \leq d-1  }}b^i_{k_{1},\ldots, k_{n} }\prod\limits_{i=1}^{n-1}y_{i}^{k_i} = 0\quad  \quad \text{ for } i = 1, \dots, n-1.
\end{equation}
Noting that $\{x_i\}_{i=1}^n$ can be computed from $\{y_i\}_{i=1,\ldots, n-1}$ via the transformation
\begin{equation}
\label{eq: transformation}
x_i=\frac{y_i}{1+y}, \quad i=1,\ldots, n-1\quad \text{and}\quad x_n=\frac{1}{1+y} \quad\text{where}\quad y=\sum\limits_{i=1}^{n-1} y_i.
\end{equation}
Thus finding an internal equilibrium of a $d$-player $n$-strategy evolutionary game using the replicator dynamics is equivalent to finding a positive root of the system of polynomial equations \eqref{eq: eqn for fitnessy}. It is noteworthy   that \eqref{eq: transformation} is precisely the transformation to obtain the Lotka–Volterra equation for $n-1$ species from the replicator dynamics for $n$ strategies~\cite{hofbauer:1998mm, page2002unifying}.

\subsection{Kostlan-Shub-Smale system of random polynomials} 
\label{sec: KSS}
Kostlan-Shub-Smale \cite{Shub1993, Kostlan1993, EK95} random polynomials $\mathscr{P}_{d,m}=(P_1,\ldots,P_m)$ consist of $m$ random polynomials in $m$ variables with common degree $\d$ 
\[
P_\ell(\x)=\sum\limits_{|\j|\leq \d} a_{\j}^{(\ell)} \x^\j
\]
where
\begin{enumerate}[(i)]
\item $\j=(j_1,\ldots,j_m)\in \Nb^m$ and $|\j|=\sum_{k=1}^m j_k$,
\item $\x=(x_1,\ldots,x_m)$ and $\x^{\j}=\prod_{k=1}^m x_k^{j_k}$,
\item $a_{\j}^{(\ell)}=a^{(\ell)}_{j_1\ldots j_m}\in \R$, $\ell=1,\ldots, m$, $|\j|\leq \d$~\text{are centred random variables},
\item $\Var(a_{\j}^{(\ell)})=\begin{pmatrix}
\d\\ \j
\end{pmatrix}=\frac{\d!}{j_1!\ldots j_m! (\d-|\j|)!}$.
\end{enumerate}
In the univariate ($m=1$) case, this class of random polynomials is also known as elliptic or normal random variables.  For further details, see  Supporting Information  \ref{SI: general random polynomials}.
\subsection{From random evolutionary games to random polynomials}
\label{sec: random EGT}
As discussed in the introduction, to obtain more realistic models capturing the unavoidable uncertainty, we consider here \textit{random evolutionary games} where the payoffs entries $\alpha^i_{i_1,\ldots, i_{d-1}}$ (thus all the coefficients $\beta^i_{i_1,\ldots, i_{d-1}}$) are random variables. Suppose that $\{\beta^i_{i_1,\ldots,i_{d-1}},~ \{i_1,\ldots, i_{d-1}\}\in \A_{k_1,\ldots,k_n}\}$ are iid centred random variables with unit variance, then it follows from \eqref{eq: b's} that  \eqref{eq: eqn for fitnessy} becomes a system of random polynomial equations whose coefficients are independent centred random variables with variances 
\begin{equation}
\label{eq: mean and variance}
\mathrm{Var}(b^i_{k_1,\ldots, k_n})=\begin{pmatrix}
d-1\\ k_1,\ldots, k_n
\end{pmatrix}.
\end{equation}
In particular, if $\{\beta^i_{i_1,\ldots,i_{d-1}},~ \{i_1,\ldots, i_{d-1}\}\in \A_{k_1,\ldots,k_n}\}$ are iid standard Gaussian random variables then $\{b^i_{k_1,\ldots, k_{n}}\}$ are centred Gaussian random variables with variances given by \eqref{eq: mean and variance}.

It follows that the polynomial system determining internal equilibria in multi-player multi-strategy random asymmetric evolutionary  games is \textit{precisely} the Kostlan-Shub-Smale polynomial system. As a consequence, the number of internal equilibria in $d$-player $n$-strategy assymmetric games is equal to the number of positive roots of the Kostlan-Shub-Smale polynomial system $\mathscr{P}_{\d-1,n-1}$. 
\begin{lemma} 
\label{lem: connection}
Let $\mathcal{N}_{d,n}$ be the number of internal equilibria of $d$-player $n$-strategy asymmetric evolutionary games and $\mathscr{N}_{\d,m}$ be the number of real roots of the Kostlan-Shub-Smale polynomial system. Then
\begin{equation}
\label{eq: KSS relation}
    \mathcal{N}_{d,n}=\frac{1}{2^{n-1}}\mathscr{N}_{d-1,n-1}.
\end{equation}
\end{lemma}
It is this exact correspondence being the novelty of the present work. This connection paves the way for characterizing the statistics of the number of internal equilibria in multi-player multi-strategy random  asymmetric evolutionary games by employing existing techniques and  results from the well-established  field of random polynomials.
\subsection{Multi-player two-strategy evolutionary games}
\label{sec: 2strategy}
In this section, we focus on $d$-player two-strategy evolutionary games. In this case, \eqref{eq: eqn for fitnessy} becomes a polynomial equation of degree $d-1$
\begin{equation}
\label{eq: eqn for y}
P_d(y):=\sum\limits_{k=0}^{d-1} b_k y^k = 0,
\end{equation}
where $y=\frac{x}{1-x}$ being the ratio of the frequencies of the two strategies and for $0\leq k\leq d-1$
\begin{equation}
\label{eq: bk}
b_k=\sum_{\{i_1,\ldots,i_{d-1}\}\in \A_k}\beta_{{i_1,\ldots, i_{d-1}}}=\sum_{\{i_1,\ldots,i_{d-1}\}\in \A_k}\Big(\alpha^1_{i_1,\ldots, i_{d-1}}-\alpha^2_{i_1,\ldots,i_{d-1}}\Big),
\end{equation}
where the sums are taken over all $\begin{pmatrix}
d-1\\k
\end{pmatrix}$ sets of $\{i_1,\ldots, i_{d-1}\}\in\A_k$ with
\begin{multline}
\A_k:=\Big\{\{i_1,\ldots, i_{d-1}\}:~1\leq i_1,\ldots, i_{d-1}\in\{1,2\}\\
\text{and there are $0\leq k\leq d-1$ players using strategy $1$ in $\{i_1,\ldots, i_{d-1}\}$}\Big\}.
\end{multline}
\begin{example} We provide concrete examples of asymmetric games to demonstrate the abstract theory.
\begin{enumerate}
\item Three-player two-strategy asymmetric game ($d=3$, $n=2$), with the following payoff matrix
\begin{equation*}
\begin{blockarray}{c|cccc}\hline
$\backslashbox{Strategy}{Opposing}$ &2,2 & 1,2&2,1&1,1\\ \hline
\begin{block}{c|cccc}
  \text{1} & \alpha^1_{2,2} & \alpha^1_{1,2} & \alpha^1_{2,1} &\alpha^1_{1,1}  \\
  \text{2} & \alpha^2_{2,2} & \alpha^2_{1,2} & \alpha^2_{2,1} &\alpha^2_{1,1}  \\
 \end{block}
 \hline
\end{blockarray}
\end{equation*}
Equation \eqref{eq: eqn for y} can be rewritten as
$$
\beta_{2,2} +(\beta_{1,2}+\beta_{2,1})y+\beta_{1,1} y^2=0,
$$
where
$$
\beta_{22}=\alpha^1_{2,2}-\alpha^2_{22},\quad \beta_{1,2}=\alpha^1_{1,2}-\alpha^2_{1,2},\quad \beta_{2,1}=\alpha^1_{2,1}-\alpha^2_{2,1},\quad \beta_{1,1}=\alpha^1_{1,1}-\alpha^2_{1,1}.
$$
\item Four-player two-strategy asymmetric game ($d=4$, $n=2$), with the following payoff matrix
\begin{equation*}
\begin{blockarray}{c|cccccccc}\hline
$\backslashbox{Strategy}{Opposing}$ &2,2,2 & 1,2,2&2,1,2&2,2,1&1,1,2&1,2,1&2,1,1&1,1,1\\ \hline
\begin{block}{c|cccccccc}
  \text{1} & \alpha^1_{2,2,2} & \alpha^1_{1,2,2} & \alpha^1_{2,1,2} &\alpha^1_{2,2,1}&\alpha^1_{1,1,2}&\alpha^1_{1,2,1}&\alpha^1_{2,1,1}&\alpha^1_{1,1,1}  \\
\text{2} & \alpha^2_{2,2,2} & \alpha^2_{1,2,2} & \alpha^2_{2,1,2} &\alpha^2_{2,2,1}&\alpha^2_{1,1,2}&\alpha^2_{1,2,1}&\alpha^2_{2,1,1}&\alpha^2_{1,1,1}  \\  
  \end{block}
 \hline
\end{blockarray}
\end{equation*}
Equation \eqref{eq: eqn for y} can be rewritten as
\[
P_4(y)=\beta_{2,2,2}+\Big(\beta_{1,2,2}+\beta_{2,1,2}+\beta_{2,2,1}\Big)y+\Big(\beta_{1,1,2}+\beta_{1,2,1}+\beta_{2,1,1}\Big)y^2+\beta_{1,1,1}y^3=0,
\]
where
$$
\beta_{i,j,k}=\alpha^1_{i,j,k}-\alpha^2_{i,j,k}\quad\text{for}\quad i,j,k\in\{1,2\}.
$$
\item Three-player three-strategy asymmetric game ($d=3$, $n=3$), with the following payoff matrix
\begin{equation*}
\begin{blockarray}{c|ccccccccc}\hline
$\backslashbox{Strategy} {Opposing}$&2,2&2,3&3,2&3,3 & 1,2&1,3&2,1&3,1&1,1 \\ \hline
\begin{block}{c|ccccccccc}
  \text{1} & \alpha^1_{2,2}&\alpha^1_{2,3}&\alpha^1_{3,2}&\alpha^1_{3,3}& \alpha^1_{1,2} & \alpha^1_{1,3} &\alpha^1_{2,1}&\alpha^1_{3,1}&\alpha^1_{1,1}  \\
 \text{2} &\alpha^2_{2,2}&\alpha^2_{2,3}&\alpha^2_{3,2}&\alpha^2_{3,3} & \alpha^2_{1,2} & \alpha^2_{1,3} &\alpha^2_{2,1}&\alpha^2_{3,1}& \alpha^2_{1,1} \\
 \text{3} &\alpha^3_{2,2}&\alpha^3_{2,3}&\alpha^3_{3,2}&\alpha^3_{3,3} & \alpha^3_{1,2} & \alpha^3_{1,3} &\alpha^3_{2,1}&\alpha^3_{3,1}& \alpha^3_{1,1} \\
 \end{block}
 \hline
\end{blockarray}
\end{equation*}
The system \eqref{eq: eqn for fitnessy} for three-player three-strategy games is 
\begin{align*}
&\beta^1_{2,2}y_2^2+(\beta^1_{2,3}+\beta^1_{3,2})y_2+\beta^1_{3,3}+(\beta^1_{1,2}+\beta^1_{2,1})y_1y_2+(\beta^1_{1,3}+\beta^1_{3,1})y_1+\beta^1_{1,1}y_1^2=0,\\
&\beta^2_{2,2}y_2^2+(\beta^2_{2,3}+\beta^2_{3,2})y_2+\beta^2_{3,3}+(\beta^2_{1,2}+\beta^2_{2,1})y_1y_2+(\beta^2_{1,3}+\beta^2_{3,1})y_1+\beta^2_{1,1}y_1^2=0,\\
\end{align*}
where
$$
\beta^1_{i,j}=\alpha^1_{i,j}-\alpha^3_{i,j},\quad \beta^2_{i,j}=\alpha^2_{i,j}-\alpha^3_{i,j}\quad \text{for}\quad i,j\in\{1,2,3\}.
$$
\end{enumerate}
\end{example}
\begin{remark}
In Section \ref{sec: random EGT} we  assumed that $\{\beta^i_{i_1,\ldots,i_{d-1}},~ \{i_1,\ldots, i_{d-1}\}\in \A_{k_1,\ldots,k_n}\}$ are iid. We call this condition (A). Recalling from \eqref{eq: beta} that $\beta^{i}_{i_1,\ldots,i_{d-1}}:=\alpha^{i}_{i_1,\ldots,i_{d-1}}-\alpha^{n}_{i_1,\ldots,i_{d-1}}$, where $\alpha^{i}_{i_1,\ldots,i_{d-1}}$ are the payoff entries. It would be more biologically interesting to assume that $\{\alpha^i_{i_1,\ldots,i_{d-1}},~ \{i_1,\ldots, i_{d-1}\}\in \A_{k_1,\ldots,k_n}\}$ are iid. We call this condition (B). Under Condition (B), Condition (A) clearly holds for $n=2$. For $n>2$, it  holds only under quite restrictive conditions such as $\alpha^{n}_{k_{1}, ..., k_{n} }$ is deterministic or $\alpha^{i}_{k_{1}, ..., k_{n} }$ are essentially identical. It is a challenging open problem to work under the general condition (B) for $n>2$.
\end{remark}
\section{Statistics of the number of internal equilibria}
\label{sec: statistics}
\subsection{The expected number of internal equilibria}
\label{sec: expected}
\begin{theorem}[The expected number of internal equilibria] 
\label{thm: expected number}
 Suppose that $\{\beta^{i}_{k_{1},\ldots, k_{n-1}}\}$ are iid standard Gaussian random variables. Then 
the expected number of internal equilibria is 
\begin{equation}
\label{eq: expected number}
\mathbb{E}(\mathcal{N}_{d,n})=\frac{1}{2^{n-1}}(d-1)^{\frac{n-1}{2}}.
\end{equation}
\end{theorem}
\begin{proof}
The statement follows directly from Lemma \ref{lem: connection} and \cite[Theorem 3.3 \& Corollary 3.4]{Kostlan1993}, see also \cite{EK95, Shub1993} and Section \ref{SI: Kac-Rice formula} in the SI for further information.
\end{proof}
\subsection{The variance of the number of internal equilibria}
\label{sec: variance}
\begin{theorem}[Asymptotic formula for the variance of the number of internal equilibria]
\label{thm: variance}
Suppose that $\{\beta^{i}_{k_{1},\ldots, k_{n-1}}\}$ are iid standard Gaussian random variables. Then it holds that
\begin{equation}
\label{eq: variance}
\lim_{d\to\infty}\frac{4^{n-1}\rm{Var}(\mathcal{N}_{d,n})}{(d-1)^\frac{n-1}{2}}=V_\infty^2,
\end{equation}
where $0<V_\infty<\infty$ is an explicit constant. Furthermore, $\mathcal{N}_{d,n}$ satisfies a central limit theorem, that is
\begin{equation}
\label{eq: CLT}
\frac{4^{n-1}\mathcal{N}_{d,n}-(d-1)^{\frac{n-1}{2}}}{(d-1)^\frac{n-1}{4}}
\end{equation}
converges in distribution, as $d\rightarrow \infty$, to a normal random variable with positive variance.  
\end{theorem}
\begin{proof}
The asymptotic of the variance and the central limit theorem of $\mathcal{N}$ follow directly from Lemma \ref{lem: connection} and \cite{armentano2018asymptotic} and \cite{armentano2021central}, respectively (see also \cite{dalmao2015asymptotic}). We refer to Section \ref{SI: variance} in the SI for further information, in particular for the explicit formula of $V_\infty$.
\end{proof}

\subsection{The distribution of the number of internal equilibria for $d$-player two-strategy games}
We provide an analytical formula for the probability that a  $d$-player two-strategy asymmetric evolutionary game has a certain number of internal equilibria. We use the following notations for the elementary symmetric polynomials
\label{sec: distribution}
\def\bs{\boldsymbol{\sigma}}
\begin{align*}
\sigma_0(y_1,\ldots,y_n)&=1,\nonumber
\\\sigma_1(y_1,\ldots,y_n)&=y_1+\ldots+y_n,\nonumber
\\\sigma_2(y_1,\ldots,y_n)&=y_1y_2+\ldots+y_{n-1}y_n
\\&\vdots\nonumber
\\\sigma_{n-1}(y_1,\ldots,y_n)&=y_1y_2\ldots y_{n-1}+\ldots+y_2y_3\ldots y_n,\nonumber
\\ \sigma_{n}(y_1,\ldots,y_n)&=y_1\ldots y_n;\nonumber
\end{align*}
and denote
\begin{equation*}
\Delta(y_1,\ldots,y_n)=\prod_{1\leq i<j\leq n}|y_i-y_j|
\end{equation*}
the Vandermonde determinant.
The main result of this section is the following theorem
\begin{theorem}
\label{thm: pm}
Suppose that the random variables $b_0,b_1,\ldots, b_{d-1}$ defined in \eqref{eq: bk} have a joint density $p(a_0,\ldots,a_{d-1})$. Then the probability that a $d$-player two-strategy asymmetric random evolutionary game has $m$ ($0\leq m\leq d-1$) internal equilibria is
\begin{equation*}
p_{m}=\sum_{k=0}^{\lfloor \frac{d-1-m}{2}\rfloor}p_{m,2k,d-1-m-2k},
\end{equation*}
where $p_{m,2k,d-1-m-2k}$ is given by
 \begin{align*}
p_{m,2k,d-1-m-2k}&=\frac{2^{k}}{m! k! (d-1-m-2k)!}\int_{\R_+^m}\int_{\R_-^{d-1-2k-m}} 
\int_{\R_+^k}\int_{[0,\pi]^k}\int_{\R}
\\& \qquad \, r_1\ldots r_k\, p(a\sigma_0,\ldots,a\sigma_{d-1}) |a|^{d-1}\Delta\, da\,d\alpha_1\ldots d\alpha_k dr_1\ldots dr_k dx_1\ldots dx_{d-1-2k}.
\end{align*}
When $\{\beta_{i_1,\ldots,i_{d-1}}\}$ are iid normal Gaussian random variables, $p_{m,2k,d-1-m-2k}$ can be expressed as
\begin{multline*}
\label{eq: pm2kG}
p_{m,2k,d-1-m-2k}=\frac{2^{k}}{m! k! (d-1-m-2k)!}~\frac{ \Gamma\Big(\frac{d}{2}\Big) }{(\pi)^{\frac{d}{2}}\prod\limits_{i=0}^{d-1}\delta_i^\frac{1}{2}}\int_{\R_+^m}\int_{\R_-^{d-1-2k-m}} 
\int_{\R_+^k}\int_{[0,\pi]^k}\, r_1\ldots r_k\\  \left(\sum\limits_{i=0}^{d-1}\frac{\sigma_i^2}{\delta_i}\right)^{-\frac{d}{2}}\Delta\,\, d\alpha_1\ldots d\alpha_k dr_1\ldots dr_k dx_1\ldots dx_{d-1-2k}.
\end{multline*}
In the above formula, $\delta_i=\begin{pmatrix}
d-1\\i
\end{pmatrix}$ and $\sigma_i$, for $i=0,\ldots,d-1$, and $\Delta$ are given by
\begin{align*}
&\sigma_j=\sigma_j(x_1,\ldots,x_{n-2k}, r_1e^{i\alpha_1}, r_1e^{-i\alpha_1},\ldots,r_k e^{i\alpha_k}, r_k e^{-i \alpha_k}),
\\&\Delta=\Delta(x_1,\ldots,x_{n-2k}, r_1e^{i\alpha_1}, r_1e^{-i\alpha_1},\ldots,r_k e^{i\alpha_k}, r_k e^{-i \alpha_k}).\end{align*}
In particular, the probability that a $d$-player two-strategy random evolutionary game has the maximal number of internal equilibria is:
\begin{equation*}
p_{d-1}=\frac{1}{(d-1)!}~\frac{\Gamma\Big(\frac{d}{2}\Big) }{(\pi)^\frac{d}{2} \prod\limits_{i=0}^{d-1}\delta_i^\frac{1}{2}}~\int_{\R_+^{d-1}} \left(\sum\limits_{i=0}^{d-1}\frac{\sigma_i^2(x_1,\ldots, x_{d-1})}{\delta_i}\right)^{-\frac{d}{2}}\Delta(x_1,\ldots, x_{d-1})\,dx_1\ldots dx_{d-1}.
\end{equation*}
\end{theorem}
\begin{proof}
The proof of this Theorem is presented in Section \ref{SI: distrution} of the Supporting Information.
\end{proof}
In Figure \ref{fig:anal_vs_simulations_pm} we compute the probability of having a certain number of internal equilibria for some small games using the analytical formulae given in Theorem \ref{thm: pm} and compare it with results from extensive numerical simulation by sampling the payoff matrix entries. The comparison shows a close correspondence between the theoretical and numerical results. 
\begin{figure}
    \centering
    \includegraphics[width=\linewidth]{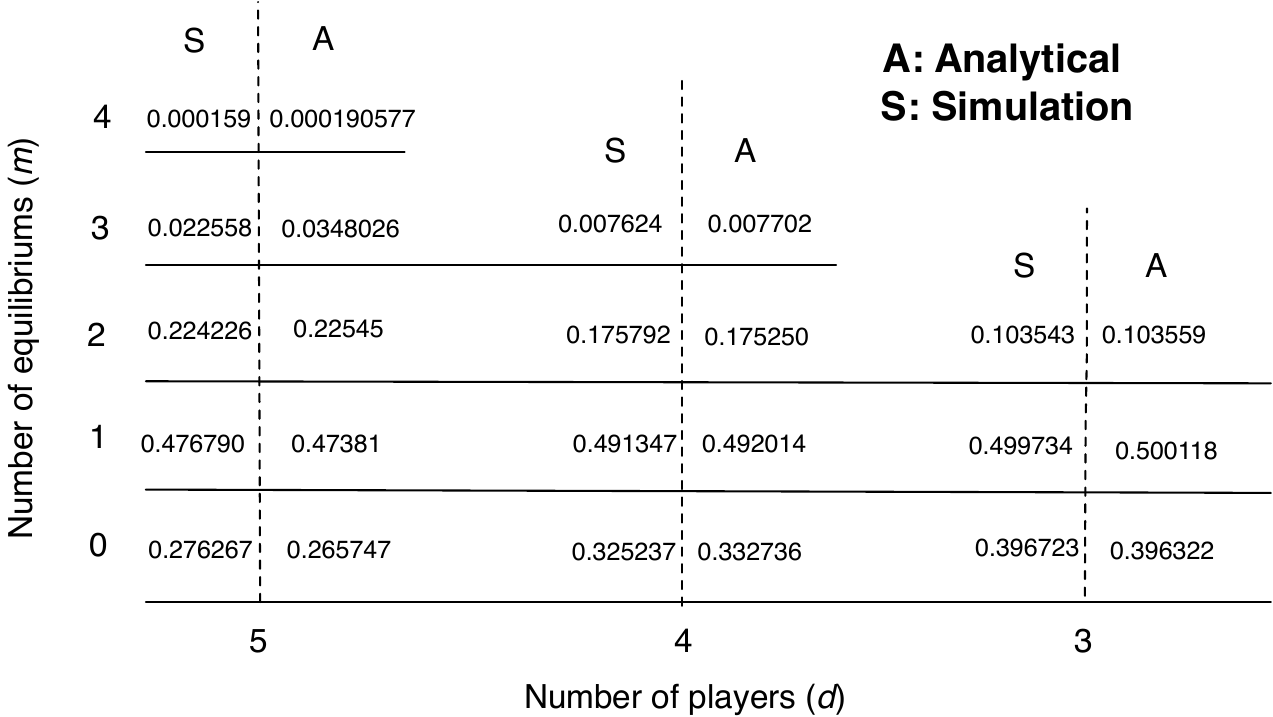}
    \caption{\textbf{Numerical calculations versus simulations  of the probability of having a concrete number ($m$) of internal equilibria, $p_m$, for different values of $d$}. Analytical results are obtained from analytical formulas in Theorem \ref{thm: pm}. Simulation results are obtained based on sampling $10^6$ payoff matrices. Analytical and simulations results are closely in accordance with each other. All results are obtained using Mathematica.} 
    \label{fig:anal_vs_simulations_pm}
\end{figure}

\begin{figure}
    \centering
    \includegraphics[width=\linewidth]{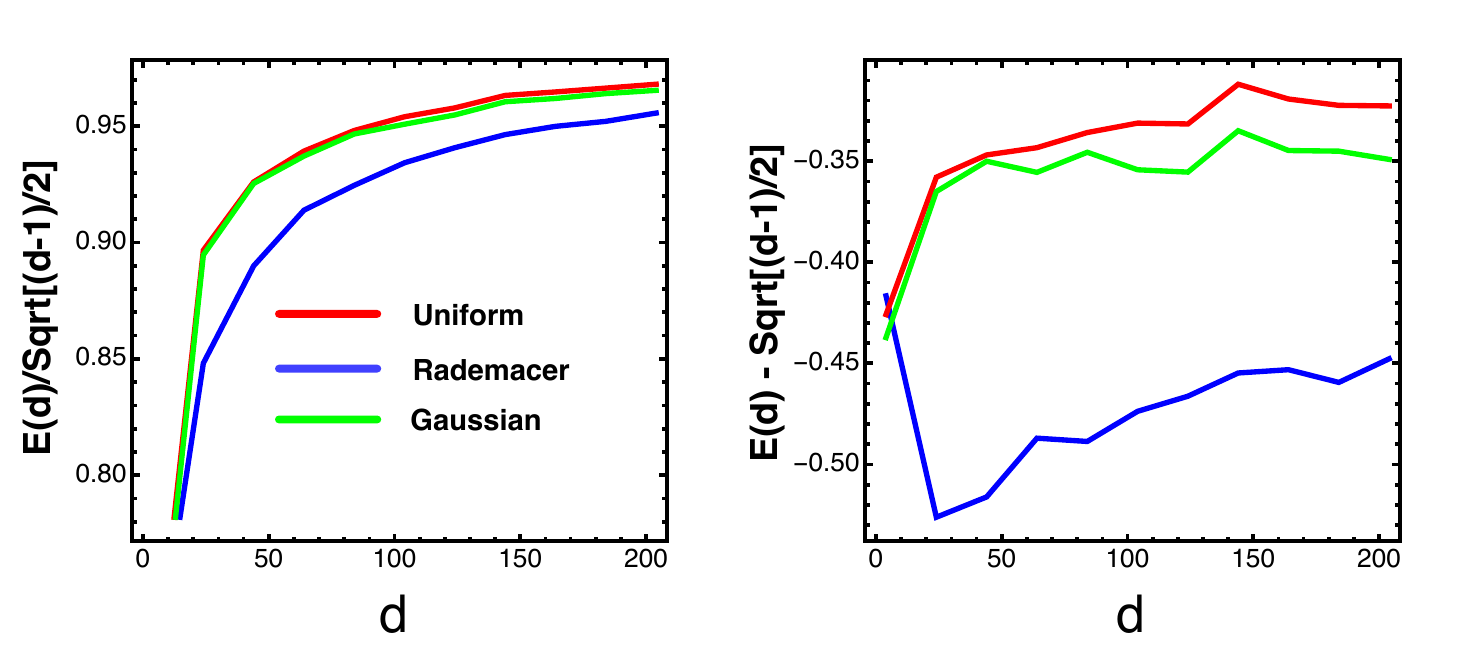}
    \caption{\textbf{Universality phenomena.  Simulation results for varying  $d$ for different distributions (Gaussian, Uniform, Rademacer) corroborate analytical results}. All simulation results are obtained based on sampling $10^6$ payoff matrices from the corresponding distributions.  All results are obtained using Mathematica.} 
    \label{fig: universal}
\end{figure}

\subsection{Universality phenomena}
In Sections \ref{sec: expected} and \ref{sec: variance}, we assume that the random coefficients $\beta_i$ are standard normal distributions. Direct applications of recent results in random polynomial theory allow us to remove this assumption, obtaining universality phenomena that characterize the asymptotic behaviour of the expected value and variance of the number of internal equilibria for $d$-player two-strategy games for a large class of general distributions. We recall from Section \ref{sec: 2strategy} that finding an internal equilibrium for a $d$-player two-strategy random asymmetric evolutionary  game amounts to finding a positive root of the random polynomial \eqref{eq: eqn for y} with coefficients $b_k$ determined from the payoff entries via \eqref{eq: bk}. From this formula, suppose that $\beta_{i_1,\ldots, i_{d-1}}$ are iid random variable, then $b_k$ is again a centered random variable with variance $\begin{pmatrix}
    d-1\\k
\end{pmatrix}$. Thus, we can write $b_k$ as 
\begin{equation}
\label{eq: bk2}
b_k=\sqrt{\begin{pmatrix}
d-1\\k
\end{pmatrix}}\, \xi_k,    
\end{equation}
where $\xi_k$ is a centered random variable with variance $1$.
\begin{theorem}[Universality for the expected number of internal equilibria]
Suppose that the random variables $\{\xi_k\}$ are independent with mean $0$, variance $1$ and finite $(2+\varepsilon)$-moment for some $\varepsilon>0$. Then
\begin{equation*}
\mathbb{E}(\mathcal{N}_{d,2})=\frac{\sqrt{d-1}}{2}+O((d-1)^{1/2-c}),
\end{equation*}
for some $c>0$ depending only on $\varepsilon$.
\end{theorem}
\begin{proof}
This is a direct consequence of Lemma \ref{lem: connection} and \cite{TaoVu14} (see also \cite{bleher2004correlations,flasche2020real,NguyenVu2022}, in particular \cite{NguyenVu2022} for a stronger statement where the assumptions on the random variables $\{\xi_k\}$ are relaxed.  
\end{proof}
\section{Symmetric vs asymmetric evolutionary games}
\label{sec: sym vs asym}
In previous works \cite{DH15,DuongHanJMB2016,DuongTranHanJMB,can2022expected}, we studied the statistics of the number of internal equilibria for $d$-player two-strategy random \textit{symmetric} evolutionary games, in which the payoff of a player in a group interaction is independent of the ordering of its members. In this symmetric case, instead of \eqref{eq: eqn for y} with coefficients given by \eqref{eq: bk2}, we obtain a different class of random polynomial 
$$
P^{\rm{sym}}(y)=\sum_{k=0}^{d-1}\begin{pmatrix}
    d-1\\k
\end{pmatrix} \xi_k y^k.
$$
Let $\mathcal{N}^{\rm{sym}}_{d,2}$ be the number of internal equilibria for $d$-player two-strategy symmetric games when the coefficient $\xi_k$ are iid standard Gaussian random variables. Then \cite{can2022expected} establishes a lower bound for the expected value of $\mathcal{N}^{\rm{sym}}_{d,2}$ for all $d$,
\begin{equation}
\label{eq: lower bound symmetric}
\mathbb{E}(\mathcal{N}^{\rm{sym}}_{d,2})\geq \frac{\sqrt{d-1}}{2} \quad \text{for all}\quad d>1,    
\end{equation}
and its asymptotic behaviour as $d\rightarrow+\infty$ 
\begin{equation}
\label{eq: leading order symm}
\mathbb{E}(\mathcal{N}^{\rm{sym}}_{d,2})=\sqrt{\frac{d-1}{2}}(1+o(1))\quad \text{as}\quad d\rightarrow +\infty.    
\end{equation}
The lower bound \eqref{eq: lower bound symmetric} is \textit{precisely} the expected number of internal equilibria for $d$-player two-strategy asymmetric games obtained in \eqref{eq: expected number}. This has an interesting biological interpretation: symmetry increases the expected number of internal equilibria (and hence, the biological or behavioural diversity).

The lower-bound \eqref{eq: lower bound symmetric} is actually true for a more general class of symmetric random polynomials. In fact, consider a general random polynomial of the form
$$
p_n(x)=\sum_{i=0}^n a_i \xi_i x^i,
$$
where the coefficients $\{a_i\}_{i=0}^n$ are symmetric, that is $a_i=a_{n-i}$ for $i=0,\ldots, n$; and $\{\xi_i\}$ are standard iid random variables. Let
\[
M_n(x)=\mathrm{var}(p_n(x))=\sum_{i=0}^n a_i^2 x^{2i}.
\]
Let $\mathbb{E}_n$ be the expected number of positive roots of $p_n$. Then we have the following theorem (see its proof in Section \ref{SI: proof general result} in  Supporting Information).
\begin{theorem}
\label{thm: general thm}
Suppose that the polynomial $\sum_{i=0}^n a_i^2 x^{i}$ has $n$ real roots. Then
\begin{equation*}
\mathbb{E}_n\geq \frac{\sqrt{n}}{2},
\end{equation*}
where the equality holds when $a_i=\sqrt{\begin{pmatrix}
    n\\i
\end{pmatrix}}$ for $i=0,\ldots, n$.
\end{theorem}
This theorem reveals an intriguing link between the expected number of real roots of a random polynomial to the real-rootedness of the associated deterministic variance polynomial. The real-rootedness of a deterministic polynomial is an active research topic with a long history. In Section \ref{SI: proof general result} in  Supporting Information, we summarize relevant results in the literature that provide necessary and sufficient conditions for a polynomial to have all real roots, including a characterization via Pólya frequency sequences (the Aissen-Schoenberg-Whitney Theorem \cite{Aissen1951}) and connections to the Lorentzian polynomials recently developed by Huh (a Fields medalist in 2022) and co-authors. We also refer the reader to for instance \cite{branden2015unimodality} for further information. As a direct consequence, Theorem \ref{thm: general thm} implies that $d$-player two strategy asymmetric games have the least expected number of equilibria among all games in which the associated polynomials satisfy the stated assumption.

In Figure \ref{fig: universal}, we numerically  compute the asymptotic behaviour of $\mathbb{E}(\mathcal{N}^{\rm{sym}}_{d,2})$ for three most popular classes of distribution
\begin{enumerate}[(i)]
\item ${\xi}_i$ are iid standard Gaussian distributions,
    \item ${\xi}_i$ are iid Rademacher distributions (i.e., receiving discrete values either $+1$ or $-1$ with equal probability $1/2$),
    \item $\{\xi\}_i$ are uniformly distributed on $[-1,1]$.
\end{enumerate}
We observe that, as $d\rightarrow +\infty$, the leading order of $\mathbb{E}(\mathcal{N}^{\rm{sym}}_{d,2})$ is the same, which is $\sqrt{\frac{d-1}{2}}$ as in \eqref{eq: leading order symm}, in all cases; while the next order term is uniformly bounded but with different bounds for different distributions. This is similar to the elliptic random polynomials arising from asymmetric games \cite{DoNguyenVu2015}. We conjecture that universality phenomenon and central limit theorem also hold true for symmetric games. 
\section{Conclusion and discussion}
In summary, we have established  an appealing connection between evolutionary game theory and random polynomial theory: the class of random polynomials arising from the study of equilibria of random asymmetric  evolutionary games is exactly the celebrated Kostlan-Shub-Smale system of random polynomials. The connection has enabled us to immediately obtain the statistics of the number of internal equilibria of random asymmetric  evolutionary games. As a consequence of our analysis, we have also proved that symmetry increases biological diversity. Furthermore, we have also numerically observed universality properties for the number of internal equilibria of symmetric random evolutionary games. Rigorously proving the universality phenomenon and a central limit theorem for symmetric games is a challenging open problem for future work. Our work also opens the door for further discoveries on the links between the two well-established theories, of evolutionary game theory and  of random polynomials,  for more complicated dynamics, such as the replicator-mutator equation.

\label{sec: summary}
\section*{Acknowledgements}
T.A.H. is supported by EPSRC (grant EP/Y00857X/1) and the Future of Life Institute. M.H.D is supported by EPSRC (grant EP/Y008561/1) and a Royal International Exchange Grant IES-R3-223047. 
\newcommand{\etalchar}[1]{$^{#1}$}

\newpage
\setcounter{section}{0} 
\renewcommand*{\thesection}{\Alph{section}}
\section*{Supporting Information}
In this Supporting Information (SI) we review relevant results on the theory of random polynomials, as well as present technical results and detailed calculations that appear in the main text.
\section{Random polynomials and system of random polynomials} 
\label{SI: general random polynomials}
The most general form of a uni-variate random polynomial of degree $N$ is given by
\begin{equation}
\Pp_{n,\xi}(x)=\sum_{i=0}^{n} \xi_i\,x^i,
\end{equation}
where $\xi_i$ are random variables. The most three well-known classes of polynomials are \cite{BS86,TaoVu14}
\begin{enumerate}[(i)]
\item Kac polynomials: $\mathrm{Var}(\xi_i)=1$,
\item Weyl (or flat) polynomials: $\mathrm{Var}(\xi_i)=\Big(\frac{1}{i!}\Big)^2$,
\item Elliptic (or binomial) polynomials: $\mathrm{Var}(\xi_i)=\begin{pmatrix}
N\\
i
\end{pmatrix}$.
\end{enumerate}
Kostlan-Shub-Smale \cite{Shub1993, Kostlan1993, EK95} polynomials are extensions of elliptic polynomials to the multivariate case. They consist of systems $\P=(P_1,\ldots,P_m)$ of $m$ polynomials in $m$ variables with common degree $d>1$ 
\[
P_\ell(\x)=\sum\limits_{|\j|\leq d} a_{\j}^{(\ell)} \x^\j
\]
where
\begin{enumerate}
\item $\j=(j_1,\ldots,j_m)\in \Nb^m$ and $|\j|=\sum_{k=1}^m j_k$;
\item $\x=(x_1,\ldots,x_m)$ and $\x^{\j}=\prod_{k=1}^m x_k^{j_k}$;
\item $a_{\j}^{(\ell)}=a^{(\ell)}_{j_1\ldots j_m}\in \R$, $\ell=1,\ldots, m$, $|\j|\leq d$; and finally
\item $\Var(a_{\j}^{(\ell)})=\begin{pmatrix}
d\\ \j
\end{pmatrix}=\frac{d!}{j_1!\ldots j_m! (d-|\j|)!}$.
\end{enumerate}
\section{Kac-Rice formula for the expected number of real roots}
\label{SI: Kac-Rice formula}
In this section, we discuss the Kac-Rice formula for computing the expected number of real roots of a random polynomial.

Let $p_{t,x,y}$ be the joint probability density function of $P_{n,\xi}(t)$ and its derivative with respect to $t$, $P'_{n,\xi}(t)$. The celebrated Kac-Rice formula for the expected number of real roots of $P_{n,\xi}$ in the interval $(a,b)$ is given by \cite{Kac1943}:
\begin{equation}
\label{eq: Kac formula}
\mathbb{E}(N_{n,\xi}(a,b))=\int_a^b f_{n,\xi}(t)\,dt,
\end{equation}
where the density function $f_{n,\xi}(t)$ is given by
$$
f_{n,\xi}(t)=\int_{-\infty}^\infty |y|p(t,0,y)\,dy.
$$

In the Gaussian case (that is, when the random variables $\{\xi_i\}_i$ are Gaussians), the density function can be computed explicitly as, see e.g. \cite[Section 2]{can2022expected}
\begin{equation}
\label{eq: Kac formula 2}
f_{n,\xi}(t)=\frac{1}{\pi}\frac{\sqrt{A_n(t)M_n(t)-B^2_n(t)}}{M_n(t)}=\frac{1}{\pi}\sqrt{\frac{1}{4t}\Big(t\frac{M_n'(t)}{M_n(t)}\Big)'},
\end{equation}
or equivalently in a logarithmic derivative form 
\begin{equation}
\label{eq: Kac formula 3}
f_{n,\xi}(t)=\frac{1}{\pi}\left[\frac{\partial^2}{\partial x\partial y}\left(\log v(x)^T C v(y)\right)\Big|_{y=x=t}\right]^\frac{1}{2},
\end{equation}
where in \eqref{eq: Kac formula 2} 
$$
A_n(t)=\mathrm{var}(P'_{n,\xi}(t)),\quad B_n(t)=\mathrm{cov}(P_{n,\xi}(t)P'_{n,\xi}(t)), \quad M_n(t)=\mathrm{var}(P_{n,\xi}(t)),
$$
and in \eqref{eq: Kac formula 3} 
\begin{equation*}
v(x)=\begin{pmatrix}
1\\
x\\
\vdots\\
x^{d-1}
\end{pmatrix}.
\end{equation*}
For Kostlan-Shub-Smale polynomials, the density function can be found explicitly, see for instance \cite{EK95}
\[
f_{n,\xi}(t)=\pi^{-\frac{m+1}{2}}\Gamma\Big(\frac{m+1}{2}\Big)\frac{(d-1)^\frac{m}{2}}{(1+\|t\|^2)^\frac{m+1}{2}},
\]
from which, the expected number of real roots follows directly from the Kac-Rice formula
\[
\mathbb{E}=d^{m/2}.
\]
\section{The distribution of the number of roots}
\label{SI: distrution}
In this section, we present an analytical formula for finding the distribution of the number of roots of a random polynomial and apply it to the random polynomial arising from symmetric random games (Theorem \ref{thm: pm}). To this end, we use the following notations for the elementary symmetric polynomials
\begin{align}
\sigma_0(y_1,\ldots,y_n)&=1,\nonumber
\\\sigma_1(y_1,\ldots,y_n)&=y_1+\ldots+y_n,\nonumber
\\\sigma_2(y_1,\ldots,y_n)&=y_1y_2+\ldots+y_{n-1}y_n\label{eq: sym pols}
\\&\vdots\nonumber
\\\sigma_{n-1}(y_1,\ldots,y_n)&=y_1y_2\ldots y_{n-1}+\ldots+y_2y_3\ldots y_n,\nonumber
\\ \sigma_{n}(y_1,\ldots,y_n)&=y_1\ldots y_n;\nonumber
\end{align}
and denote
\begin{equation}
\label{eq: Delta}
\Delta(y_1,\ldots,y_n)=\prod_{1\leq i<j\leq n}|y_i-y_j|.
\end{equation}
the Vandermonde determinant. The following theorem provides an analytical formula for the probability $p_{m,2k,n-m-2k}$ that $\Pp$ has $m$ positive, $2k$ complex and $n-m-2k$ negative zeros.
\begin{theorem}(\cite[Theorem 5.1]{DuongTranHanJMB}
\label{theo: Zap06}
Assume that the random variables $\xi_0,\xi_1,\ldots, \xi_n$ have a joint density $p(a_0,\ldots,a_n)$. Let $0\leq m\leq d-1$ and $0\leq k\leq \lfloor \frac{n-m}{2}\rfloor$. The probability $p_{m,2k,n-m-2k}$ that $\Pp$ has $m$ positive, $2k$ complex and $n-m-2k$ negative zeros is given by
\begin{multline}
\label{eq: pm2k}
p_{m,2k,n-m-2k}=\frac{2^{k}}{m! k! (n-m-2k)!}\int_{\R_+^m}\int_{\R_-^{n-m-2k}} 
\int_{\R_+^k}\int_{[0,\pi]^k}\int_{\R}
\\ r_1\ldots r_k p(a\sigma_0,\ldots,a\sigma_{n}) |a^{n}\Delta|\, da\,d\alpha_1\ldots d\alpha_k dr_1\ldots dr_k dx_1\ldots dx_{n-2k},
\end{multline}
where 
\begin{align}
\label{eq:sigma}
&\sigma_j=\sigma_j(x_1,\ldots,x_{n-2k}, r_1e^{i\alpha_1}, r_1e^{-i\alpha_1},\ldots,r_k e^{i\alpha_k}, r_k e^{-i \alpha_k}),
\\&\Delta=\Delta(x_1,\ldots,x_{n-2k}, r_1e^{i\alpha_1}, r_1e^{-i\alpha_1},\ldots,r_k e^{i\alpha_k}, r_k e^{-i \alpha_k}).\label{eq:Delta}
\end{align}
As consequences,
\begin{enumerate}[(1)]
\item The probability that $\Pp$ has $m$ positive zeros is
\begin{equation}
p_{m}=\sum_{k=0}^{\lfloor \frac{n-m}{2}\rfloor}p_{m,2k,n-m-2k}.
\end{equation}
\item In particular, the probability that $\Pp$ has the maximal number of positive zeros is
\begin{equation}
p_{n}=\frac{2^{k}}{k! (n-2k)!}\int_{\R_+^{n}}\int_{\R}p(a\sigma_0,\ldots,a\sigma_{n})\, |a^{n}\,\Delta|\, dadx_1\ldots dx_{n},
\end{equation}
where
\begin{equation*}
\sigma_j=\sigma_j(x_1,\ldots,x_{n}),\quad\Delta=\Delta(x_1,\ldots,x_{n}).
\end{equation*}
\end{enumerate}
\end{theorem}
We now apply this theorem to the random polynomial \eqref{eq: eqn for y} to obtain an explicit formula for the probability $p_m$ that a $d$-player two-strategy symmetric random evolutionary game has $m$ ($0\leq m\leq d-1$) internal equilibria. Due to the special property of \eqref{eq: eqn for y}, the formula \eqref{eq: pm2k} will be simplified. 

\def\bs{\boldsymbol{\sigma}}

The following theorem is Theorem \ref{thm: pm} in the main text. 
\begin{theorem}
\label{thm: distribution}
The probability that a $d$-player two-strategy random evolutionary game has $m$ ($0\leq m\leq d-1$) internal equilibria is
\begin{equation}
p_{m}=\sum_{k=0}^{\lfloor \frac{d-1-m}{2}\rfloor}p_{m,2k,d-1-m-2k},
\end{equation}
where $p_{m,2k,d-1-m-2k}$ is given by
\begin{multline}
\label{eq: pm2kG}
p_{m,2k,d-1-m-2k}\\=\frac{2^{k}}{m! k! (d-1-m-2k)!}~\frac{ \Gamma\Big(\frac{d}{2}\Big) }{(\pi)^{\frac{d}{2}}\prod\limits_{i=0}^{d-1}\delta_i^{1/2}}\int_{\R_+^m}\int_{\R_-^{d-1-2k-m}} 
\int_{\R_+^k}\int_{[0,\pi]^k}\, r_1\ldots r_k\, \left(\sum\limits_{i=0}^{d-1}\frac{\sigma_i^2}{\delta_i}\right)^{-\frac{d}{2}}\Delta\\d\alpha_1\ldots d\alpha_k dr_1\ldots dr_k dx_1\ldots dx_{d-1-2k}
\end{multline}
where $\sigma_i$, for $i=0,\ldots,d-1$, and $\Delta$ are given in \eqref{eq:sigma}--\eqref{eq:Delta} and
$\delta_i=\begin{pmatrix}
d-1\\i
\end{pmatrix}$.

In particular, the probability that a $d$-player two-strategy random evolutionary game has the maximal number of internal equilibria is
\begin{equation*}
p_{d-1}=\frac{1}{(d-1)!}~\frac{\Gamma\Big(\frac{d}{2}\Big) }{(\pi)^\frac{d}{2} \prod\limits_{i=0}^{d-1}\delta_i^{1/2}}~\int_{\R_+^{d-1}}\left(\sum\limits_{i=0}^{d-1}\frac{\sigma_i^2}{\delta_i}\right)^{-\frac{d}{2}}\Delta \,dx_1\ldots dx_{d-1},
\end{equation*}
Note that in formulas above, $\sigma_j=\sigma_j(x_1,\ldots,x_{d-1}),\quad \Delta=\Delta(x_1,\ldots,x_{d-1})$.
\end{theorem}
\begin{proof}[Proof of Theorem \ref{thm: distribution}]
\def\y{\mathbf{y}}
The proof of this theorem follows the same lines as that of \cite[Theorem 5.2]{DuongTranHanJMB}.

Since $\{\beta_{i_1,\ldots,i_{d-1}}\leq j\leq d-1\}$ are i.i.d. standard normally distributed, the coefficients $\beta_k$ in \eqref{eq: eqn for y} are independent Gaussians with mean zero and variance $ \begin{pmatrix}
d-1\\k
\end{pmatrix}$. Therefore, their joint distribution $p(y_0,\ldots,y_{d-1})$ is given by
\[
p(y_0,\ldots,y_{d-1})=\frac{1}{(2\pi)^{\frac{d}{2}} \prod_{i=0}^{d-1}\begin{pmatrix}
d-1\\i
\end{pmatrix}^\frac{1}{2}}\exp\left[-\frac{1}{2}\sum_{i=0}^{d-1}\frac{y_i^2}{\begin{pmatrix}
d-1\\i
\end{pmatrix}}\right]=\frac{1}{(2\pi)^{\frac{d}{2}}|\C|^\frac{1}{2}}\exp\Big[-\frac{1}{2}\y^T\C^{-1}\y\Big],
\]
where $\mathbf{y}=[y_0~~y_1~~\ldots~~y_{d-1}]^T$ and $\C$ is the covariance matrix given by $\C=\rm{diag}\Big(\begin{pmatrix}
d-1\\i
\end{pmatrix}\Big)_{i=0}^{d-1}$.
Therefore,
\begin{equation}
\label{eq: formula of p}
p(a\sigma_0,\ldots, a \sigma_{d-1})=\frac{1}{(2\pi)^\frac{d}{2} |\C|^\frac{1}{2}}\exp\Bigg(-\frac{a^2}{2}\bs^T\,\C^{-1}\,\bs\Bigg)\quad\text{where}\quad
\bs=[\sigma_0~\sigma_1~\ldots~\sigma_{d-1}]^T.
\end{equation}
Using the following formula for moments of a normal distribution, 
\begin{equation*}
\int_{\R}|x|^n\exp\big(-\alpha x^2\big)\,dx=\frac{\Gamma\big(\frac{n+1}{2}\big)}{\alpha^\frac{n+1}{2}},
\end{equation*}
we compute
\begin{equation*}
\int_{\R}|a|^{d-1}\exp\Bigg(-\frac{a^2}{2}\bs^T\,\C^{-1}\,\bs\Bigg)\,da=\frac{\Gamma\Big(\frac{d}{2}\Big)}{\Big(\frac{\bs^T\C^{-1}\bs}{2}\Big)^{\frac{d}{2}}}=\frac{2^\frac{d}{2}\Gamma\Big(\frac{d}{2}\Big)}{\big(\bs^T\C^{-1}\bs\big)^{\frac{d}{2}}}.
\end{equation*}
Applying Theorem \ref{theo: Zap06} to the polynomial $P$ given in \eqref{eq: eqn for y} and using the above identity we obtain
\begin{align*}
p_{m,2k,d-1-m-2k}&=\frac{2^{k}}{m! k! (d-1-m-2k)!}\int_{\R_+^m}\int_{\R_-^{d-1-2k-m}} 
\int_{\R_+^k}\int_{[0,\pi]^k}\int_{\R}
\\& \qquad \, r_1\ldots r_k\, p(a\sigma_0,\ldots,a\sigma_{d-1}) |a|^{d-1}\Delta\, da\,d\alpha_1\ldots d\alpha_k dr_1\ldots dr_k dx_1\ldots dx_{d-1-2k}
\\&=\frac{2^{k}}{m! k! (d-1-m-2k)!}~\frac{1}{(2\pi)^\frac{d}{2} |\C|^\frac{1}{2}}~ 2^{\frac{d}{2}}\Gamma\Big(\frac{d}{2}\Big) ~\int_{\R_+^m}\int_{\R_-^{d-1-2k-m}} 
\int_{\R_+^k}\int_{[0,\pi]^k}
\\& \qquad \, r_1\ldots r_k\, \big(\bs^T\C^{-1}\bs\big)^{-\frac{d}{2}}~\Delta\,d\alpha_1\ldots d\alpha_k dr_1\ldots dr_k dx_1\ldots dx_{d-1-2k}
\\&=\frac{2^{k}}{m! k! (d-1-m-2k)!}~\frac{\Gamma\Big(\frac{d}{2}\Big) }{(\pi)^\frac{d}{2} |\C|^\frac{1}{2}}~\int_{\R_+^m}\int_{\R_-^{d-1-2k-m}} 
\int_{\R_+^k}\int_{[0,\pi]^k}
\\& \qquad \, r_1\ldots r_k\, \big(\bs^T\C^{-1}\bs\big)^{-\frac{d}{2}}~\Delta\,d\alpha_1\ldots d\alpha_k dr_1\ldots dr_k dx_1\ldots dx_{d-1-2k},
\end{align*}
which is the desired equality \eqref{eq: pm2kG} by definition of $\C$.
\end{proof}

\section{The variance of the number of real roots}
\label{SI: variance}
In this section, we summarise the results of \cite{armentano2018asymptotic,Dalmao2015} on the asymptotic behaviour of the variance of the number of real roots, $\N$, of the Kostlan-Sub-Smale system of $m$ random polynomials in $m$ variables with common degree $\mathbf{d}$ (See Section \ref{sec: KSS}). To this end, we need the following notations: for $k=1,\ldots,m$ let $\xi_k,\eta_k$ be independent standard normal random vectors on $\R^k$. Let us define
\begin{align*}
&\sigma^2(t) = 1 - \frac {t^2 e^{-t^2} } {1 - e^{-t^2}};\quad \tau(t) = e^{-t^2}\Big( 1 - \frac { t^2 } {1 - e^{-t^2} } \Big);\quad \rho(t) = \frac{\tau(t)}{\sigma^2(t)};
\\&  m_{k,j}=\mathbb{E}(\|\xi_k\|^j)=2^{j/2}\frac{\Gamma((j+k)/2)}{\Gamma(k/2)},~~\text{where $\|\cdot\|$ is the Euclidean norm on $\R^k$};
\\& \text{for}~ k=1,\ldots, m-1, ~~M_k(t)=\mathbb{E}\Big[\|\xi_k\|\|\eta_k+\frac{e^{-t^2/2}}{(1-e^{-t^2})^{1/2}}\xi_k\|\Big];
\\& \text{for}~k=m,~~ M_{m}(t)=\mathbb{E}\Big[\|\xi_{m}\|\|\eta_{m}+\frac{\tau(t)}{(\sigma^4(t)-\tau^2(t))^{1/2}}\xi_{m}\|\Big].
\end{align*}
\begin{theorem}[Asymptotic behaviour of the variance \cite{armentano2018asymptotic,Dalmao2015}]
It holds that
\begin{equation}
\lim_{\mathbf{d}\to\infty}\frac{\rm{Var}(\N)}{\mathbf{d}^\frac{m}{2}}=V_\infty^2,
\end{equation}
where the limiting scaled variance $V_\infty$ is given explicitly by 
\begin{equation}
\label{eq: Vinfinity}
V_\infty^2=\frac{1}{2}+\frac{\kappa_{m}\kappa_{m-1}}{2(2\pi)^{m}}\int_0^\infty t^{m-1}\Big[\frac{\sigma^4(t)(1-\rho^2(t))}{1-e^{-t^2}}\Big]^{1/2}\Big[\prod_{k=1}^{m}M_k(t)-\prod_{k=1}^{m}m_{k,1}^2\Big]\,dt,
\end{equation}
where $\kappa_m$ is the $m$-volume of the sphere $S^{m}$. Furthermore, in the case $m=1$, $\N$ satisfies a central limit theorem
\begin{equation}
\frac{\N-\E\N}{(\rm{Var}(\N))^\frac{1}{2}}\to \mathcal{N}(0,1),
\end{equation}
where $\mathcal{N}(0,1)$ is the standard normal distribution.
\end{theorem}
We also refer the reader to \cite{Azais2005, Wschebor2005} for results about the asymptotic behaviours when $m\to \infty$.
\section{Universality phenomena for the expected number of roots of random polynomials}
\label{SI: universality}
In this section, we recall the theorem of \cite{TaoVu14} on the universality of the expected number of real roots of the elliptic random polynomial.
\begin{theorem}\cite{TaoVu14}
\label{thm: universality random polynomials}
Consider elliptic random polynomials
$$
\Pp_{n,\xi}(x)=\sum_{i=0}^{n}\sqrt{\begin{pmatrix}
n\\i
\end{pmatrix}} \xi_i\,x^i,
$$
Suppose that the random variables $\{\xi_i\}_i$ are independent with mean $0$, variance $1$ and finite $(2+\varepsilon)$-moments. Then
\begin{equation}
\label{eq: universality}
\mathbb{E}(N_{\Pp}(\mathbb{R}))=\sqrt{n}+ O(n^{1/2-c}),
\end{equation}
for some $c$ depending only on $\varepsilon$.
\end{theorem}
In a more recent paper \cite{NguyenVu2022}, the authors extend, among other things, the universality result above to a more flexible condition which allows a constant number of $\xi_i$ to have non-zero means. 
\section{Polynomials with all real roots}
We consider the following polynomial with real coefficients
\[
P(x)=\sum_{i=0}^n a_i x^i.
\]
In this section, we briefly summarise relevant works in the literature that provide necessary and sufficient conditions the polynomial $P$ to have all real roots. This real-rootedness condition appears in Theorem \ref{thm: general thm} in the main text. In particular, this section presents many important polynomials that satisfy the condition of Theorem \ref{thm: general thm}. We refer the reader to the paper \cite{branden2015unimodality} for a great exposition of this interesting topics. 
\subsection{Log-concave, Pólya frequency sequence}
We say that a sequence $\{a_n\}$  is log-concave if
\[
a_n^2\geq  a_{n-1}a_{n+1}.
\]
for all $n$. More generally, given $r\in \mathbb{R}$, we say that a sequence $\{a_n\}$  is $r$-factor log-concave if \cite{Mcnamara2010}
\[
a_n^2\geq r a_{n-1}a_{n+1}.
\]
for all $n$. The log-concavity properties are intimately related to the roots of polynomials. 
The following necessary condition for $P$ to have all real roots is dated back to Newton, see \cite{Kurtz1992}
\begin{theorem}If all the roots of $P$ are real, then
\begin{equation}
\label{eq: logconcave of coeffs}
a_i^2\geq \frac{n-1+1}{n-1}\frac{i+1}{i} a_{i-1}a_{i+1}, \quad i=1,\ldots, n-1.     
\end{equation}
If the roots of $P$ are not all equal, these inequalities are strict.
\end{theorem}
Note that condition \eqref{eq: logconcave of coeffs} is equivalent to the condition that the sequence $\left\{a_i/\begin{pmatrix}
    n\\i
\end{pmatrix}\right\}_{i=0}^n$ is log-concave. The following theorem provides a sufficient condition for $P$ to have all real roots.
\begin{theorem}\cite{Kurtz1992}
\label{thm: Kurtz}
Let $P$ be a polynomial of degree $n\geq 2$ with positive coefficients. If the sequence $\{a_i\}$ is 4-factor log-concave, that is
\begin{equation}
    a_i^2-4a_{i-1}a_{i+1}>0, \quad i=1,\ldots n-1
\end{equation}
then all the roots of $P$ are real and distinct. Furthermore, that 4-factor log-concavity cannot be replaced by $(4-\varepsilon)$-factor log-concavity for any $\varepsilon>0$.
\end{theorem}
The following theorem provides a necessary and sufficient condition for $P$ to have real roots. Let $P^{(j)}$ be the $j$-th derivative of $P$.
\begin{theorem}\cite{Chamberland2020}
   The zeros of $P$ are real and distinct if and only if
  \begin{equation}
     (P^{(j)}(x))^2>P^{(j-1)}(x)P^{(j+1)}(x)  
  \end{equation} 
  for all $x\in \mathbb{R}$, $j=1,\ldots, n-1$.
\end{theorem}
However, the condition is practically hard to verify. Aissen-Schoenberg-Whitney Theorem \cite{Aissen1951} offers an alternative characterization of a polynomial with real roots. To state this theorem, we recall the concept of a Pólya frequency sequence. A sequence of real numbers $(a_k)_{k=0}^\infty$ is called a \textit{Pólya frequency (or PF) sequence} if the infinite matrix $(a_{j-i})_{i,j=0}^\infty$ is totally positive, (i.e., all its minors have a nonnegative determinant) where we adopt the convention that $a_k=0$ for $k<0$. A finite sequence $(a_0,\ldots, a_n)$ is called a PF sequence if the infinite sequence $(a_0,\ldots, a_n, 0,\ldots)$ is a PF.

PF sequences are characterized by the following theorem \cite{Edrei1953}.
\begin{theorem}
    A sequence $\{a_k\}_{k=0}^\infty\subset \mathbb{R}$ of real numbers is PF if and only if its generating function may be expressed as
    \[
    \sum_{k=0}^\infty a_k x^k=C x^m e^{ax} \prod_{k=0}^\infty (1+ \alpha_k x)\Big/\prod_{k=0}^\infty (1- \beta_k x),
    \]
    where $C,a\geq 0$, $m\in \mathbb{N}$, $\alpha_k, \beta_k\geq 0$ for all $k\in \mathbb{N}$, and $\sum_{k=0}^\infty (\alpha_k+\beta_k)<\infty$.
\end{theorem}
We also refer the reader to \cite{Wang2005} for examples of linear transformations that preserve PF property.



The connection between finite PF sequences and the zeros of the corresponding polynomials is given by the following fundamental Aissen-Schoenberg-Whitney theorem \cite{Aissen1951}.
\begin{theorem}[PF characterization]
Let $a_0,\ldots, a_n\geq 0$. Then
\[
(a_0,\ldots, a_n)\quad\text{is a PF sequence}\quad \longleftrightarrow \sum_{i=0}^n a_n x^n\quad\text{has only real zeros}.
\]
\end{theorem}

\subsection{Operations that preserve real-rootednesss}
In this section we discuss operations on polynomials that preserve the real-rootedness property. Obviously the following operations preserve the real-rootedness of a polynomial.
\begin{enumerate}
    \item Differentiation: If $p(x)$ is real-rooted, so is $p'(x)$ (by Rolle’s theorem).
    \item Reciprocation: If $p(x)$ is real-rooted, then so is the reciprocal polynomial $r(x) = x^n p(1/x)$.
\end{enumerate}
The following results provide more operations.
\begin{proposition}\cite{Driver2007}, \cite{branden2011iterated}
\begin{enumerate}
    \item if $\sum_{k=0}^n a_kx^k$ has only real zeros, then $\sum_{k=0}^n \frac{a_k}{k!}x^k$ has also only real zeros.
    \item if $\sum_{k=0}^n a_kx^k$ and $\sum_{k=0}^n b_kx^k$ have only real zeros. Then $\sum_{k=0}^n i! a_i b_i x^i$ has only real zeros. 
    \item suppose that the polynomial $\sum_{k=0}^n a_k x^k$ has only real and negative zeros. Then so does the polynomial
    \[
    \sum_{k=0}^n (a_k^2-a_{k-1}a_{k+1}) x^k
    \]
    where $a_{-1}=a_{n+1}=0$.
\end{enumerate}   
\end{proposition} 
Furthermore, the following polynomials have only real (negative) roots \cite{Driver2007}
\begin{enumerate}
    \item $\sum_{k=0}^n\begin{pmatrix}
        n\\k
    \end{pmatrix}^2 \begin{pmatrix}
        2k\\k
    \end{pmatrix} x^k$.
    \item $
    \sum_{k=0}^n \begin{pmatrix}
        n\\k
    \end{pmatrix}\frac{(a_1)_k(a_2)_k\ldots (a_p)_k}{(b_1)_k(b_2)_k\ldots(b_q)_k}x^k
    $
    where $a_i, b_j>0$ and $(\alpha)_k=\alpha(\alpha+1)\ldots(\alpha+k-1)$ denotes the is the Pochhammer symbol
    \item Narayana polynomials \cite{branden2006linear}
    \[
    \sum_{k=0}^n \Big(\begin{pmatrix}
        n\\k
    \end{pmatrix}^2-\begin{pmatrix}
        n\\k-1
    \end{pmatrix}\begin{pmatrix}
        n\\k+1
    \end{pmatrix}\Big) x^k
    \]
    \end{enumerate}
Note that binomial coefficients are log-concave. In fact, we have 
\[
\begin{pmatrix}
    n\\i
\end{pmatrix}^2- \begin{pmatrix}
    n\\i-1
\end{pmatrix}\begin{pmatrix}
    n\\i+1
\end{pmatrix}=\begin{pmatrix}
    n\\i-1
\end{pmatrix}\begin{pmatrix}
    n\\i+1
\end{pmatrix}\frac{n+1}{i(n-i)}\geq 0.
\]
\subsection{Stable, Lorentzian polynomials}
Recent breakthroughs on Lorentzian polynomials developed by Brändén and Huh offer deeper connections between stability and real-rootedness properties of a polynomial. We follow the their seminal paper \cite{branden2020lorentzian}.

We recall that a polynomial $p$ in $\mathbb{R}[x_1,\ldots,x_n]$ is stable if $p$ is non-vanishing on $\mathcal{H}^n$ or identically zero, where $\mathcal{H}$ is the open upper half plane in $\mathbb{C}$. Let $S_n^d$ be the set of degree $d$ homogeneous stable polynomials in $n$ variables with non-negative coefficients.

When $p\in S_n^d$, the stability of $p$ is equivalent to any one of the following statements on univariate polynomials in the variable $x$
\begin{enumerate}[-]
    \item For any $u\in \mathbb{R^n}_{>0}$, $p(xu-v)$ has only real zeros for all $v\in \mathbb{R}^n$.
    \item For some $u\in \mathbb{R^n}_{>0}$, $p(xu-v)$ has only real zeros for all $v\in \mathbb{R}^n$.
    \item For any $u\in \mathbb{R^n}_{\geq 0}$ with $p(u)>0$, $p(xu-v)$ has only real zeros for all $v\in \mathbb{R}^n$.
    \item For some $u\in \mathbb{R^n}_{\geq 0}$ with $p(u)>0$, $p(xu-v)$ has only real zeros for all $v\in \mathbb{R}^n$.
\end{enumerate}
In the above statements, we want the univariate polynomial $Q(x):=p(xu-v)$ to have only real roots. The following provide example such a polynomial~\cite{branden2020lorentzian}[Example 2.3]: consider the homogeneous bivariate polynomial with positive coefficients
\[
p(x,y)=\sum_{k=0}^n a_k x^k y^{n-k}.
\]
Then $p$ is stable if and only if the univariate polynomial
\[
Q(x)=p(x,1)=\sum_{k=0}^n a_k x^k
\]
has only (non-positive) real roots.

The paper \cite{borcea2010multivariate} provide a general characterization of a real stable polynomial with two-variables. Let $p[x,y]$  be of degree n(not necessary homogeneous). Then $p$ is real stable if and only if there exist two $n\times n$ positive semi-definite matrices $A, B$ and a symmetric $n\times n$ matrix $C$ such that
\[
p(x,y)=\pm\det(xA+yB+C).
\]
Finally, according to \cite{branden2020lorentzian}, any polynomial in $S^d_n$ is Lorentzian.
\section{Proof of Theorem \ref{thm: general thm}}
\label{SI: proof general result}
\begin{proof}[Proof of Theorem \ref{thm: general thm}]
The proof of this theorem can be obtained by directly adapted the proof of \cite[Theorem 1.1, part (1)]{can2022expected}. 

Let $r_1,\ldots, r_n$ be $n$ real roots of the polynomial $\sum_{i=0}^n a_i^2 x^{i}$. Obviously all of them are negative.  
We can write the polynomial $M_n$ as
\begin{equation}
\label{eq: representation of Mn}
M_n(t)=m_n \prod_{k=1}^n (t^2+r_k),
\end{equation}
where $m_n$ is the leading coefficient. Using the representation \eqref{eq: representation of Mn} of $M_n$ we have
\begin{align*}
M_n'(t)=2t m_n\sum_{k=1}^n\prod_{j\neq k} (t^2+r_j), \quad \frac{M_n'(t)}{M_n(t)}=\sum_{k=1}^n\frac{2t}{t^2+r_k},\quad 
\Big(t\frac{M_n'(t)}{M_n(t)}\Big)'=\sum_{k=1}^n\frac{4 t r_k}{(t^2+r_k)^2}.
\end{align*}
Hence, according to \eqref{eq: Kac formula 2}, the density function can be represented as
\begin{align}
\label{eq: f in terms of roots of Mn}
f_n(t)^2=\frac{1}{4t}\Big(t\frac{M_n'(t)}{M_n(t)}\Big)'=\sum_{k=1}^n\frac{r_k}{(t^2+r_k)^2}.
\end{align}
From  \eqref{eq: f in terms of roots of Mn}  and using Cauchy-Schwartz inequality we have
\begin{equation}
 \label{lowerbound}   
 \left(\sum_{k=1}^n\frac{\sqrt{r_k}}{t^2+r_k}\right)^2\leq n\sum_{k=1}^n\frac{r_k}{(t^2+r_k)^2}=n f_n(t)^2,
\end{equation}
In addition, if all $r_k=r$ are the same then the inequality becomes an equality,
\[
f_n(t)=\sqrt{n}\frac{\sqrt{r}}{t^2+r}
\]
From  \eqref{lowerbound} we deduce
$$
f_n(t)\geq \frac{1}{\sqrt{n}}\sum_{k=1}^n\frac{\sqrt{r_k}}{t^2+r_k}.
$$
Therefore,
\begin{equation}
\label{lower-bound-of-E}
\mathbb{E}(N_n)=\frac{1}{\pi}\int_{-\infty}^\infty f_n(t)\,dt\geq \frac{1}{\sqrt{n}}\sum_{k=1}^n\int_{-\infty}^\infty\frac{\sqrt{r_k}}{\pi(t^2+r_k)}\,dt=\sqrt{n}.   
\end{equation}
when all $r_k$ are equal then $\mathbb{E}(N_n)=\sqrt{n}$.
Therefore we have the following characterisation of a class of random polynomials satisfying this equality: suppose that
$$
a_i^2=\sigma^2 \begin{pmatrix}
    n\\i
\end{pmatrix} a^{n-i}
$$
for some $\sigma, a>0$, that is we consider the class of random polynomial of the form
\[
p_n(x)=\sum_{i=0}^n \sqrt{\begin{pmatrix}
    n\\i
\end{pmatrix}} a^{\frac{n-i}{2}}\xi_i
\]
where $\{\xi_i\}$ are independent $\mathcal{N}(0,\sigma).$ 
Then $\mathbb{E}(N_n)=\sqrt{n}$. In fact, in this case
$$
M_n(x)=\sigma^2\sum_{i=0}^n \begin{pmatrix}
    n\\i
\end{pmatrix} a^{n-i} x^i= \sigma^2 (x+a)^n.
$$
This has all real (negative) roots, which are all equal to $-a$. We can even take $a=a_n$ sine the last integral in \eqref{lower-bound-of-E} does not depend on $r_k$.
\end{proof}
\end{document}